\documentclass[11pt]{article}
\usepackage[T1]{fontenc}
\usepackage{a4wide}
\usepackage{lipsum}

\title{Facial diagrams and cycle double cover}

\author{
Babak Ghanbari\thanks{Computer Science Institute of Charles University, Prague, Czech Republic.  E-mail: {\tt babak@iuuk.mff.cuni.cz}. Supported by the project GAUK182623 of the Charles University Grant Agency.}
\and
Robert Šámal\thanks{Computer Science Institute of Charles University,  Prague, Czech Republic. E-mail: {\tt samal@iuuk.mff.cuni.cz}. 
Supported by grant 25-16627S of the Czech Science Foundation.}
}

\makeatletter

\newcommand{\shorttitle}{\@title}

\def\@maketitle{%
  \newpage
  \begin{center}%
  \let \footnote \thanks
    {\small Proceedings of the 13th European Conference on Combinatorics, Graph Theory and Applications\\ EUROCOMB'25\\
    Budapest, August 25 - 29, 2025
    }
    \vskip 0.5em
    \rule{\linewidth}{0.04cm}
    \vskip 3.5em
    {\LARGE \textbf{\textsc{\@title}} \par}%
    \vskip 1.5em
    {\textbf{\textsc{(Extended abstract)}} \par}
    \vskip 2.5em%
    {\large
      \lineskip .5em%
      \begin{tabular}[t]{c}%
        \@author
      \end{tabular}\par}%
  \end{center}%
  \par
  }
\makeatother

\usepackage{graphicx}
\usepackage{amsthm}
\usepackage{amssymb}
\usepackage{amsmath}
\usepackage{xcolor}
\usepackage{todonotes}
\usepackage{xfrac}

\theoremstyle{plain} 
\newtheorem{theorem}{Theorem}
\newtheorem{lemma}{Lemma}
\newtheorem{proposition}{Proposition}
\newtheorem{observation}{Observation}

\newtheorem{conjecture}{Conjecture}

\numberwithin{theorem}{section}
\numberwithin{definition}{section}
\numberwithin{lemma}{section}
\numberwithin{corollary}{section}
\numberwithin{proposition}{section}
\usepackage{cite}
\usepackage{bbm}

\begin{document}

\thispagestyle{empty}
\maketitle

\begin{abstract}
We approach the cycle double cover conjecture by looking for a circular 2-cell embedding of cubic graphs on an arbitrary surface. 
It is easy to see that if such an embedding exists, we can get to it from an arbitrary starting 2-cell embedding by repeating ``twists of an edge''. 
We study this twisting operation in detail and deduce bounds on the number of singular edges (edges where a face meets itself). 
\end{abstract}

\section{Introduction}
The cycle double cover conjecture claims that for every bridgeless graph there is a collection of cycles covering each edge exactly twice. 
See \cite{zhang_2012} for details and partial results. It is known that it is enough to show the conjecture for 3-regular graphs and that 
for such graphs it is equivalent to finding a circular 2-cell embedding on some surface. 
In other words, we are looking for an embedding such that the dual graph is loopless. Our goal in this paper is to 
represent the dual graph in a way that makes it easier to study the effects of a single twist operation. 

In an embedding of a graph $G$, an edge $e$ is called \emph{singular} if there exists a facial walk $F$ that traverses $e$ twice. Otherwise, $e$ is called \emph{regular}.
If a singular edge $e$ is traversed twice in the same direction by $F$ we call it \emph{good singular}, and if $e$ traversed twice in opposite directions by
$F$ we call it \emph{bad singular}. Note that in orientable embeddings all singular edges are bad singular. 
For graph embeddings, we follow the notation of~\cite{Mohar}, in particular we will use pair $(\pi, \lambda)$ to denote the combinatorial embedding of a graph. 

\section{Twist operation}

We study the operation on graph embeddings called the \textbf{twist} of an edge~$e$, i.e., the change of $\lambda(e)$. 
The following proposition shows the effect of twisting an edge  common to two faces, see Figure~\ref{fig:Twist-Join}. 
Our goal in this paper is to study this in more detail with the motivation by the Cycle Double Cover conjecture. 


\begin{proposition}\label{prop1}
Let $F_1$ and $F_2$ be distinct facial walks and let $e\in~E(F_1)~\cap~E(F_2)$ be regular. 
Then, by the twist of $e$, $F_1$ and $F_2$ will become one facial walk. 
\end{proposition}

\begin{figure}
    \centering
    \includegraphics[width=70mm,scale=0.5]{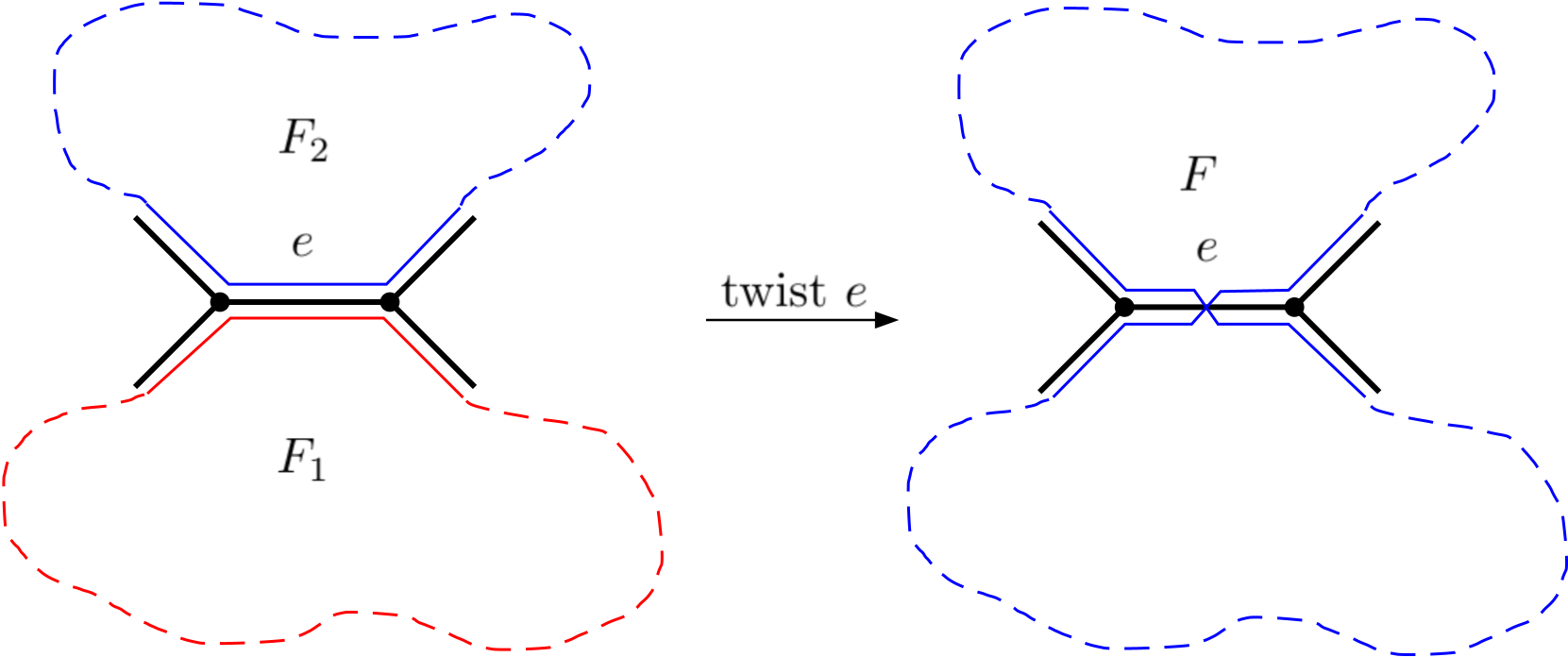}
    \caption{Twist of an edge $e$ between facial walks $F_1$ and $F_2$}
    \label{fig:Twist-Join}
\end{figure}

As mentioned in \cite{Mohar}, each local change of rotation around a vertex $v$ changes the signs of all  
three edges incident with $v$, which is the same as the twist of all three edges incident with $v$. Therefore, we have the following lemma.

\begin{lemma}\label{twist-lemma}
Let $(\pi,\lambda)$ be an embedding of a bridgeless cubic graph $G$ and let $(\pi', \lambda)$ be another embedding resulting from $(\pi, \lambda)$ by a sequence of local changes of rotations. Then, one can find an embedding $(\pi,\lambda')$ with a sequence of local twists of the edges in $(\pi, \lambda)$ such that $(\pi, \lambda')$ and $(\pi', \lambda)$ have the same set of facial walks.
\end{lemma}

In other words, one can start from any embedding in the graph and find all the other embeddings of the graph by twisting the edges. 
So, if the CDC conjecture is true, then given any embedding $(\pi,\lambda)$ of a bridgeless cubic graph $G$, we can always find some (actually, every) CDC by a
sequence of twist operations on $(\pi,\lambda)$.

\section{Facial diagram}
In this section, we introduce another representation of facial walks that helps us in our study of the twist operation; we call it \emph{facial diagram}.
Let $(\pi, \lambda)$ be an embedding for a bridgeless cubic graph $G$ with facial walks $F_1, F_2, \dots, F_k$ where
$F_i = \{(v^i_0, e^i_1, v^i_1, e^i_2, \dots, v^i_{t-1}, e^i_t)\}$.
A \emph{facial diagram}, or \emph{FD} for $G$ is a cubic graph $H$ with its vertices (here we say node) and edges defined as follows.
$H$ has nodes $V(H) = \{e^i_j: e^i_j \in F_i\}$. Since an embedding covers every edge of $G$ twice, we have $|V(H)| = 2|E(G)|$. $H$ has the following three types of edges. (See Figures~\ref{fig:FD} and~\ref{fig:Example-FD}.) 
\begin{enumerate}
    \item \textbf{Facial link}: $v^i_j = (e^i_j, e^i_{j+1})$ where $e^i_j, v^i_j, e^i_{j+1} \in F_i$. These are exactly the copies of the vertices of $G$ that appear on the facial walks of the embedding.
    \item \textbf{Singular link}: $(e^i_j, e^i_k)$ where 
    $e^i_j, e^i_k \in F_i$
    and $e^i_j$ and $e^i_k$ are the same edge in $G$ (but different nodes in $H$). 
    \item \textbf{Regular link}: $(e^i_k, e^j_l)$ where $e^i_k = e^j_l$ as edges of $G$, 
    $e^i_k \in F_i$, and $e^j_l \in F_j$.
\end{enumerate}
Since there is a 1-1 correspondence between the set of edges of $G$ and the set of singular and regular links of $H$, in the rest of the paper, if there is no confusion, we simply write a singular link $(e^i_j, e^i_k) \in F_i$ as $e_j$ (its label in E(G)) and a regular link $(e^i_k, e^j_l)$ between $F_i$ and $F_j$ as $e_k$ (corresponding edge label in $G$); See Figure~\ref{fig:Example-FD} as an example. We say that two singular links $e_1$ and $e_2$ are \emph{crossing}, if they appear as $(\dots, e_1,\dots, e_2, \dots, e_1, \dots, e_2, \dots)$ in a facial walk $F$ of the facial diagram.

\begin{figure}
    \begin{center}
        \includegraphics[scale=0.18]{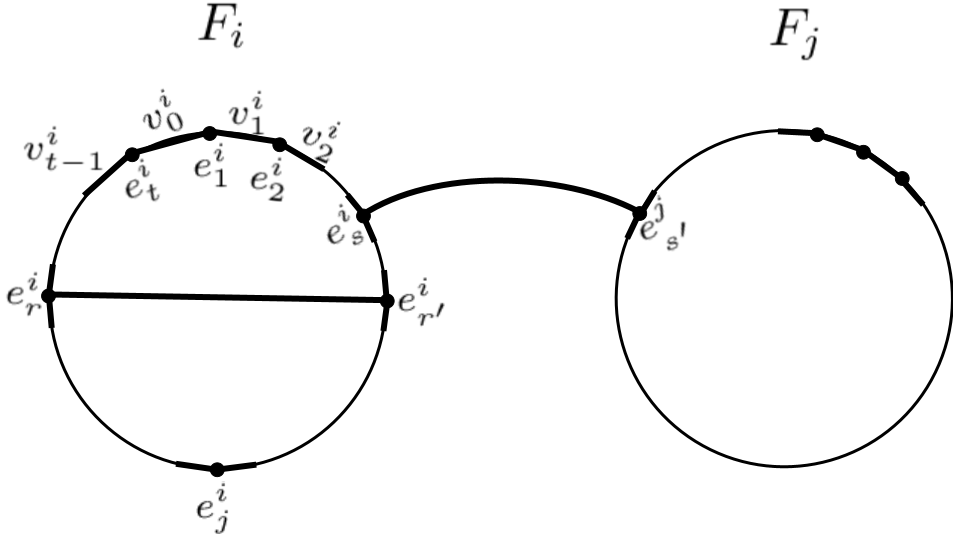}
    \end{center}
    \caption{Representation of facial walks $F_i$ and $F_j$ in the facial diagram of a graph $G$ with a singular link $(e^i_r, e^i_{r'})$ and a regular link $(e^i_s, e^j_{s'})$. } 
    \label{fig:FD}
\end{figure}

Next, we define a signature for every singular link in $H$. Let $e = (a, b)$ be a singular edge in a facial walk $F$. If $e$ is a bad singular edge i.e. $(\dots, a, e, b, \dots, b, e, a, \dots)\in F$, then we give the singular link $e$ in $H$ a $+$ sign and we call it a \emph{bad singular link}. Otherwise, if $e$ is a good singular edge i.e. $(\dots, a, e, b, \dots, a, e, b, \dots)\in F$, then we give the singular link $e$ in $H$ a $-$ sign and we call it a \emph{good singular link}. See the following pictures.
\begin{center}
    \includegraphics[scale=0.18]{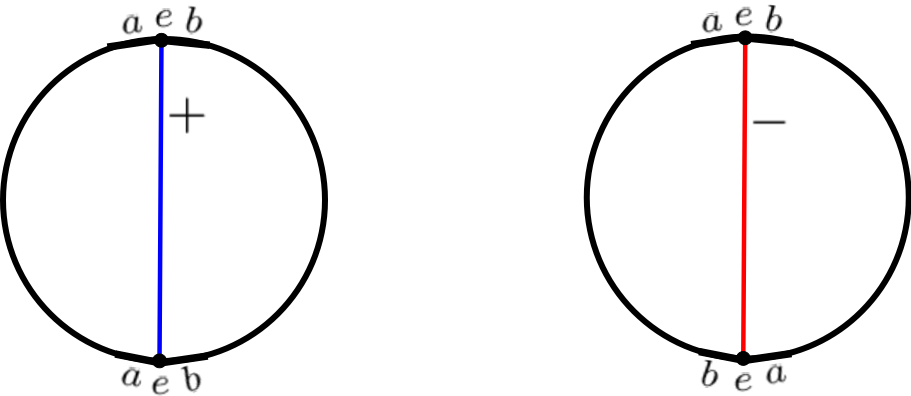}
\end{center}

\begin{figure}
    \begin{center}
        \includegraphics[scale=0.18]{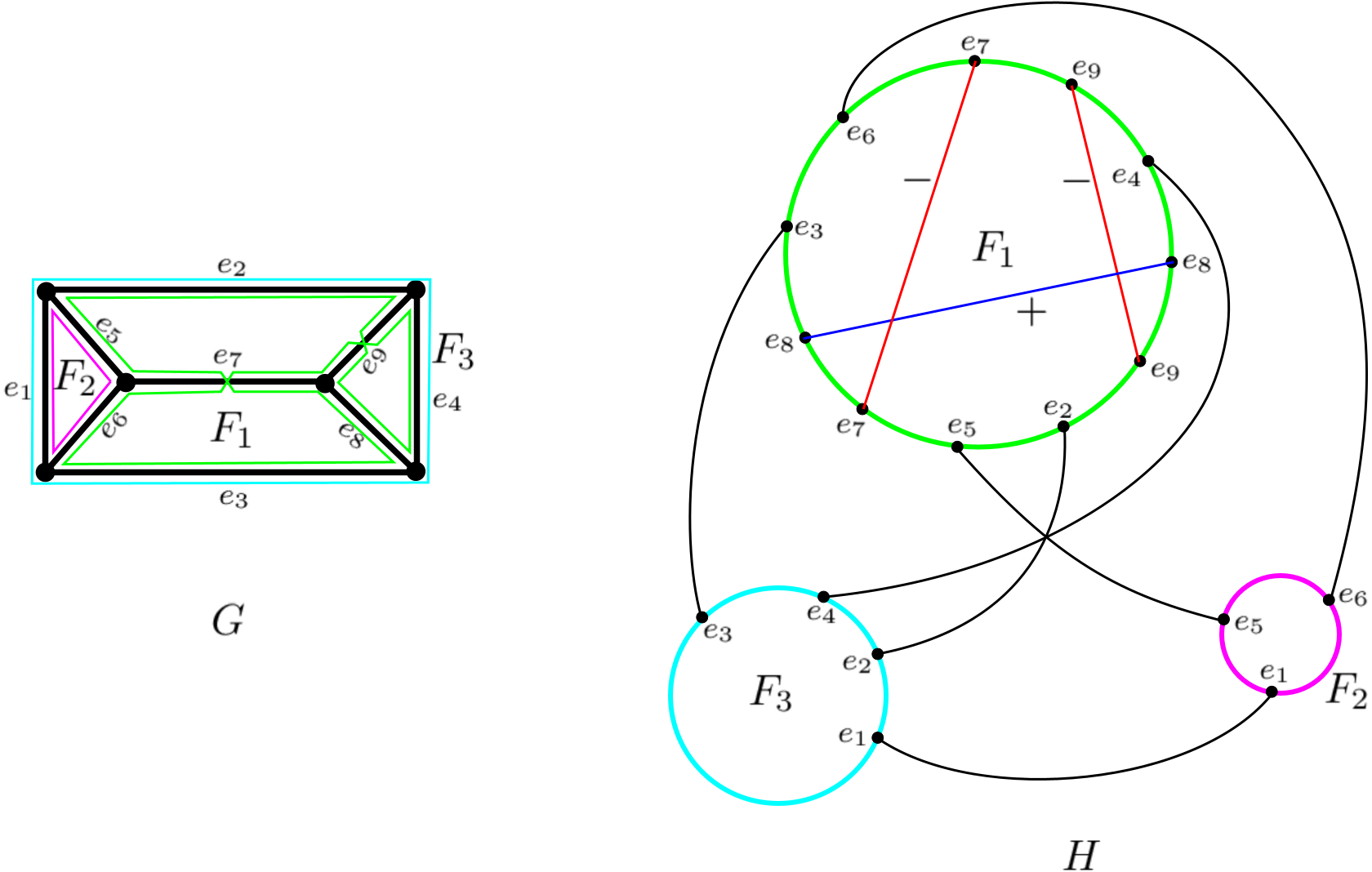}
    \end{center}
    \caption{The facial diagram $H$ of an embedding of a cubic graph $G$ where $e_8$ is a bad singular link and $e_7$ and $e_9$ are good singular links. For simplicity, facial links' labels of H and vertex labels of $G$ are omitted.}
    \label{fig:Example-FD}
\end{figure}

\begin{lemma}\label{crossing-lemma}
    Let $e_1$, and $e_2$ be two crossing singular links in a facial walk $F$.
    \begin{enumerate}
        \item If $e_1$ has $-$ sign. Then, by twist of $e_1$, both $e_1$, and $e_2$ turn to a regular link.
        \item If $e_1$ has $+$ sign. Then twist of $e_1$ changes the sign of $e_2$ while the sign of $e_1$ remains unchanged.
    \end{enumerate}
\end{lemma}

By Lemma~\ref{crossing-lemma}, if we twist a $-$ link the number of singular edges of the new corresponding embedding reduces at least by~$1$. If we twist a $+$ link, the number of singular links remains unchanged. However, after the twist of $e$, the order of nodes and links on one side of $e$ will change. We will use this property of the $+$ links to change the sign of some of the other singular links. See also Figure~\ref{fig:twist-positive}.

\begin{figure}
    \begin{center}
        \includegraphics[scale=.16]{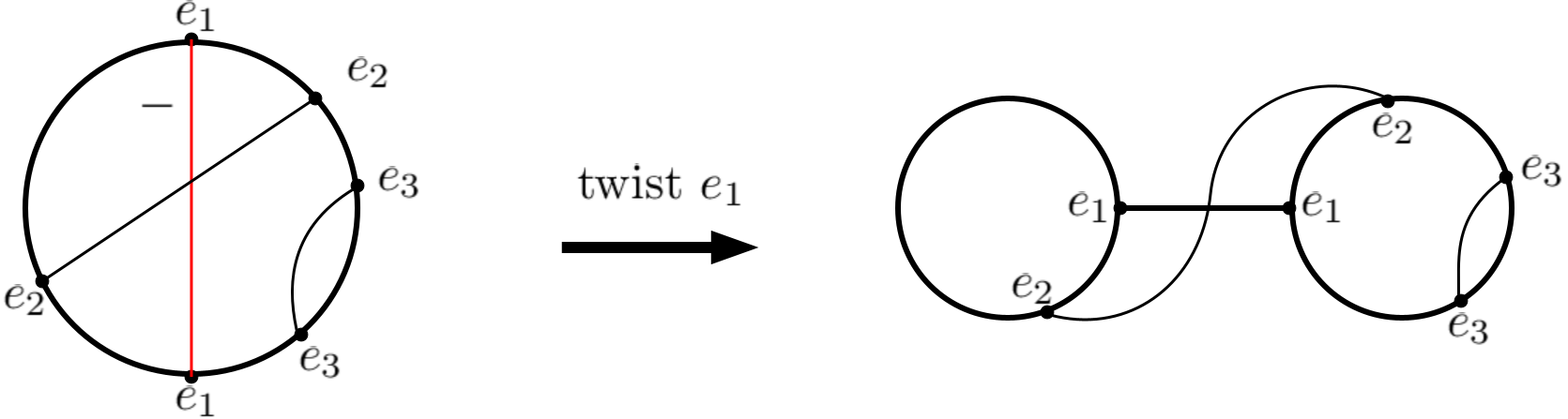}
    \end{center}
    \caption{Twist of a $-$ link $e_1$. After the twist, $e_2$ turns to regular but $e_3$ remains singular.}
    \label{fig:twist-negative}
\end{figure}

\begin{figure}
    \begin{center}
        \includegraphics[scale=.16]{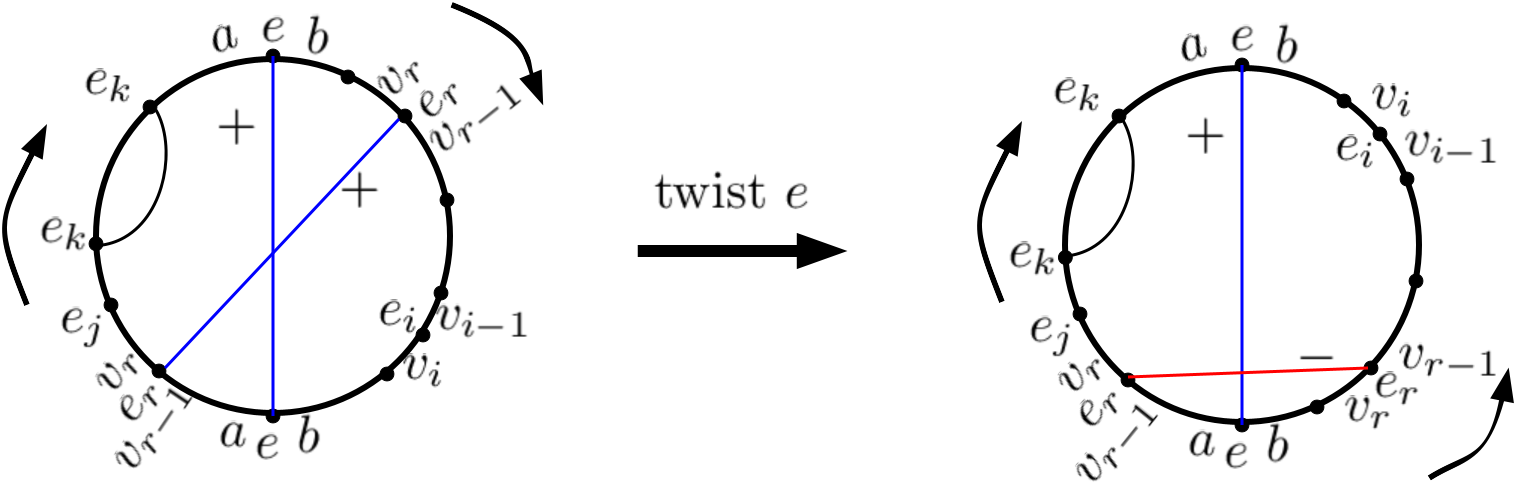}
    \end{center}
    \caption{Twist of a $+$ link $e$. Observe how the order of nodes and facial links on one side of $e$ changes. As a result sign of $e_r$ changes to $-$ and then twist of $e_r$ changes both $e$ and $e_r$ to regular edges.}
    \label{fig:twist-positive}
\end{figure}

\subsection{Properties of the facial diagram}
The facial diagram has the following nice properties. We include their proofs in the full paper. Let $G$ be a bridgeless cubic graph,  $F$ be a facial walk in an embedding of $G$, and $H$ be a facial diagram of the same embedding. Let $e_1 = (a,b)$, $e_2 = (b,c)$, and $e_3 = (b, d) \in E(F)$. Then, 


\begin{enumerate}
    \item  $(\dots, a, e_1, b, e_2, c, \dots, a, e_1, b, e_2, c, \dots) \notin F$.
    \item $(\dots, a, e_1, b, e_2, c, \dots, c, e_2, b, e_1, a, \dots) \notin F$.
    \smash{\raisebox{-1cm}[0mm][0mm]{\hbox to 0pt{\includegraphics[width=0.45\textwidth]{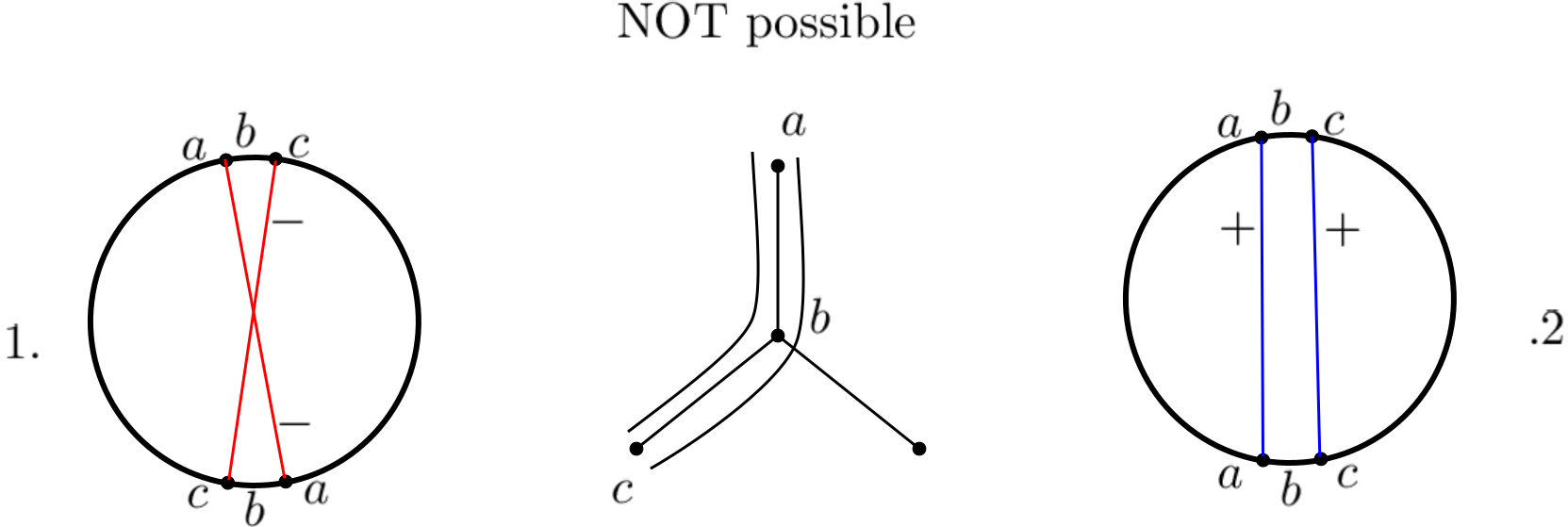}}}}
    \vspace{10mm}


    \item Suppose that $e_1$ and $e_2$ are singular and have sign $-$. If 
        $(\dots, a, e_1, b, e_2, c, \dots) \in F$ 
        then,~either 
        $(\dots, a, e_1, b, e_2, c, \dots, a, e_1, b, e_3, d, \dots, d, e_3, b, e_2, c, \dots) \in F,$
        or \newline$(\dots, a, e_1, b, e_2, c, \dots, d, e_3, b, e_2, c,\dots, a, e_1, b, e_3, d, \dots) \in F$

    \item Suppose that $e_1$ and $e_2$ are singular and have sign $+$. If 
        $\{\dots, a, e_1, b, e_2, c, \dots\} \in F$
        then,~either
        $\{\dots, a, e_1, b, e_2, c, \dots, c, e_2, b, e_3, d, \dots, d, e_3, b, e_1, a, \dots\} \in F,$ or\newline
        $\{\dots, a, e_1, b, e_2, c, \dots, d, e_3, b, e_1, a,\dots, c, e_2, b, e_3, d, \dots\} \in F$

        \begin{center}
            \includegraphics[scale=0.16]{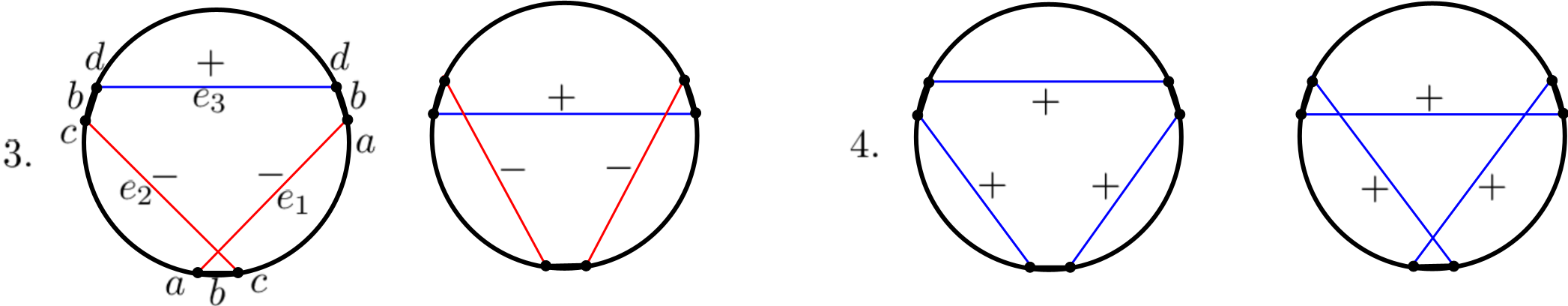} 
        \end{center}
    \item We call a vertex $v$ \emph{saturated} if all the edges incident with $v$ are singular. 
      The number of $+$ edges incident with a saturated vertex is either~$1$ or~$3$. 
    \item Consider an embedding of a bridgeless cubic graph $G$ with the minimum number of 
        singular edges. The facial diagram of this embedding has no crossing links. 
    \item Let $e \in E(F_1) \cap E(F_2)$ be a regular link in a facial diagram 
        $H$, If we twist $e$, the number of singular links in the new facial diagram $H'$ is the number of singular links in $H$ plus $|E(F_1) \cap E(F_2)|$.
    \item Let $e$ be a regular link. Then twist of $e$ is a good singular link
        and therefore has sign $-$ in the facial diagram. 
    \item Twist of a regular link $e$ does not change the sign of any singular links.
\end{enumerate}
\section{Random embedding}
In this section, we approximate the number of singular edges by using random embeddings. Let $(\pi, \lambda)$ be a random embedding of a cubic graph $G$. 
Results of several computer experiments counting the number of singular edges in a random embedding of a cubic graph indicate that the expected number of bad singular edges is $\frac{m}{3}$, where $m$ is the number of edges in $G$. 
Also, in expectation, this number seems to be the same for good singular edges and regular edges. Therefore, we propose the following conjecture.  

\begin{conjecture}\label{rnd-emb-conj}
In a random embedding of a bridgeless cubic graph $G$,
    \begin{enumerate}
        \item The expected number of bad singular edges is $\frac{m}{3}$.
        \item The expected number of good singular edges is $\frac{m}{3}$.
        \item The expected number of regular edges is $\frac{m}{3}$.
    \end{enumerate}   
\end{conjecture}

As an application, we have the following theorem.

\begin{theorem}
    Suppose that Conjecture~\ref{rnd-emb-conj} is true and let $G$ be a bridgeless cubic graph. Then, there exists an embedding of $G$ with at most $\frac{m}{3}$ singular edges where $m = |E(G)|$.
\end{theorem}
\begin{proof}
    If conjecture~\ref{rnd-emb-conj} holds, then there exists an embedding and therefore a facial diagram where the number of singular links with sign $+$ is at most $\frac{m}{3}$. If every other link is regular, there is nothing to prove and the number of singular links (and therefore singular edges of the embedding) is at most $\frac{m}{3}$. Otherwise, start with any $-$ link $e$ and apply Lemma~\ref{crossing-lemma} to the facial diagram. As a result, $e$ and every link crossing with $e$ changes to a regular link in the new facial diagram. Since it does not increase the number of $+$ links, one can repeat this process until there are no $-$ links without increasing the number of $+$ links. Therefore, there exists a facial diagram with at most $\frac{m}{3}$ bad singular links.   
\end{proof}

A referee suggested to us that the above result actually has a simpler proof: take any perfect matching, and extend the complementary 2-factor into a 
facial double cover. Then, only the edges of the matching can be singular (see \cite{Ghanbari}).

%
%
\bibliographystyle{alpha}
\bibliography{RN.bib}

\end{document}